\documentclass{amsart}

\usepackage{amsmath,amsthm,amssymb,mathtools,amscd,color}
\usepackage[all]{xy}
\usepackage{setspace}
\allowdisplaybreaks

\newtheorem{theorem}{Theorem}[section]
\newtheorem{proposition}[theorem]{Proposition}
\newtheorem{lemma}[theorem]{Lemma}
\newtheorem{corollary}[theorem]{Corollary}

\newtheorem{question}[theorem]{Problem}

\theoremstyle{definition}
\newtheorem{definition}[theorem]{Definition}

\theoremstyle{remark}
\newtheorem{remark}[theorem]{Remark}
\numberwithin{equation}{section}

\newcommand{\R}{\mathbb R}
\newcommand{\Z}{{\mathbb Z}}
\newcommand{\N}{{\mathbb N}}
\newcommand{\C}{{\mathbb C}}
\newcommand{\Q}{{\mathbb Q}}

\title[Spherical designs for quaternions and modular forms]{Spherical designs for finite quaternionic unit groups and their applications to modular forms}

\author[M.~Hirao]{Masatake Hirao}
\address[M.~Hirao]{Department of Information Science and Technology\\
	Aichi Prefectural University\\
	Nagakute-city, Aichi, 480-1198\\
	Japan}
\email{hirao@ist.aichi-pu.ac.jp}

\author[H.~Nozaki]{Hiroshi Nozaki}
\address[H.~Nozaki]{Department of Mathematics Education\\ 
	Aichi University of Education\\
	1 Hirosawa, Igaya-cho, Kariya, Aichi 448-8542\\
	Japan}
\email{hnozaki@auecc.aichi-edu.ac.jp}

\author[K.~Tasaka]{Koji Tasaka}
\address[K.~Tasaka]{Department of Mathematics\\
	Kindai University\\
	3-4-1, Kowakae, Higashiosaka, Osaka 577-8502\\
	Japan}
\email{tasaka@math.kindai.ac.jp}

\subjclass{05B30, 11P21, 11F30}
\keywords{Finite quaternionic unit group, spherical designs of harmonic index, harmonic strength, Linear programming bound, uniqueness, spherical theta functions, modular forms}

\begin{document}

\begin{abstract}
For a finite subset $X$ of the $d$-dimensional unit sphere, the harmonic strength $T(X)$ of $X$ is the set of $\ell\in \N$ such that $\sum_{x\in X} P(x)=0$ for all harmonic polynomials $P$ of homogeneous degree $\ell$.
We will study three exceptional finite groups of unit quaternions, called the binary tetrahedral group $2T$ of order 24, the octahedral group $2O$ of order 48, and the icosahedral group $2I$ of order 120, which can be viewed as a subset of the 3-dimensional unit sphere.
For these three groups, we determine the harmonic strength and show the minimality and the uniqueness as spherical designs. 
In particular, the group $2O$ is unique as a minimal subset $X$ of the 3-dimensional unit sphere with $T(X)=\{22,14,10,6,4,2 \}\cup \mathbb{O}^+$, where $\mathbb{O}^+$ denotes the set of all positive odd integers.
This result provides the first characterization of $2O$ from the spherical design viewpoint.

For $G\in \{2T,2O,2I\}$, we consider the lattice $\mathcal{O}_{G}$ generated by $G$ over $R_G$ on which the group $G$ acts on by multiplication, where $R_{2T}=\Z,\ R_{2O}=\Z[\sqrt{2}],\ R_{2I}=\Z[(1+\sqrt{5})/2]$ are the ring of integers. 
We introduce the spherical theta function $\theta_{G,P}(z)$ attached to the lattice $\mathcal{O}_G$ and a harmonic polynomial $P$ of degree $\ell$ and prove that they are modular forms. 
By applying our results on the characterization of $G$ as a spherical design, we determine the cases in which the $\C$-vector space spanned by all $\theta_{G,P}(z)$ of harmonic polynomials $P$ of homogeneous degree $\ell$ has dimension zero--without relying on the theory of modular forms.
\end{abstract}

\maketitle

\section{Introduction}

\subsection{Background}
A finite non-empty subset $X$ of the unit sphere $\mathbb{S}^{d-1}$ of the $d$-dimensional Euclidean space $ \R^d$ is a \emph{spherical $t$-design} if it holds that
\begin{equation*}\label{eq:spherical_design}
\frac{1}{|X|} \sum_{\boldsymbol{x}\in X} P(\boldsymbol{x}) = \frac{1}{|\mathbb{S}^{d-1}|} \int_{\boldsymbol{x} \in \mathbb{S}^{d-1}}P(\boldsymbol{x})d\sigma(\boldsymbol{x}) \quad \mbox{for all $P\in \R[x_1,\ldots,x_d]$ of degree $\le t$},
\end{equation*}
where $\sigma$ is the surface measure on $\mathbb{S}^{d-1}$ and $|\mathbb{S}^{d-1}|=\int_{\boldsymbol{x} \in \mathbb{S}^{d-1}}d\sigma(\boldsymbol{x})$.
This concept was first introduced by Delsarte--Goethals--Seidel in \cite{DelsarteGoethalsSeidel77}.
Since then, the problem of determining the maximum value of $t$ for a given finite subset $X$ of the unit sphere, referred to as the \emph{strength} of $X$, and the characterization of such subsets as spherical designs have been studied in connection with various algebraic structures~\cite{BannaiBannai09,BBTZ17}.

We wish to study these problems in the context of spherical designs of harmonic index, which is a slight generalization of spherical $t$-designs.
For $X\subset \R^d$, we define the \emph{harmonic strength} $T(X)$ of $X$ by
\begin{equation}\label{eq:T(X)}
T(X):=\left\{\ell \in \N \ \middle|\  \sum_{\boldsymbol{x}\in X} P(\boldsymbol{x})=0,\ \forall P\in {\rm Harm}_\ell (\R^d)\right\},
\end{equation}
where ${\rm Harm}_\ell (\R^d)$ denotes the $\R$-vector space of harmonic polynomials of homogeneous degree $\ell$.
Note that $X$ is a spherical $t$-design if and only if $1,2,\ldots,t\in T(X)$.
A finite subset $X\subset \mathbb{S}^{d-1}$ such that $T\subset T(X)$ is called a \emph{spherical $T$-design}.
This concept was introduced by Delsarte--Seidel \cite{DS89} as a spherical analogue of the design in association schemes \cite[Section 3.4]{Delsarte:PHD},
and has been actively studied in a decade (cf.~\cite{BannaiOkudaTagami15,HiraoNozakiTasaka,Miezaki13,MisawaMunemasaSawa,OkudaYu16,Pandey,ZhuBannaiBannaiKimYu17}).

Our fundamental problems are as follows.

\begin{question}\label{Q1}
\begin{itemize}
\item[(i)] Given $X\subset \R^d$, determine $T(X)$.
\item[(ii)] Given $T\subset \N$, characterize finite subsets $X\subset \R^d$ such that $T\subset T(X)$.
\end{itemize}
\end{question}

As an illustrative example of addressing Problem \ref{Q1}, consider a finite subset $X\subset \mathbb{S}^{d-1}$ that is \emph{antipodal}, meaning that $-\boldsymbol{x}\in X$ for all $\boldsymbol{x}\in X$.
In this case, the harmonic index $T(X)$ contains the set $\mathbb{O}^+$ of all positive odd integers, providing a partial solution to (i).
Regarding (ii), a finite subset $X\subset \mathbb{S}^{d-1}$ must necessarily be antipodal if $\mathbb{O}^+\subset T(X)$ (cf.~\cite{MisawaMunemasaSawa}).

Remark that Problem \ref{Q1} (i) is closely connected to the theory of modular forms through spherical theta functions, which was first pointed out by Venkov \cite{Venkov84}.
Regarding Problem \ref{Q1} (ii), we note that for given $t,d\in \N$, the existence of spherical $t$-designs was shown in \cite{SZ84}. 
Bondarenko--Radchenko--Viazovska \cite{BondarenkoRadchenkoViazovska} further showed that for each $N\ge c_d t^{d-1}$, there exists a spherical $t$-design in $\mathbb{S}^{d-1}$ consisting of $N$ points, where $c_d$ is a constant depending only on $d$.

In this paper, we mainly consider a set of certain quaternions as $X$ in Problem \ref{Q1}.
Let us briefly describe it.
Denote by $\mathbb{H}$ the $\R$-algebra of (Hamilton's) \emph{quaternion} with generators $\{i,j\}$ and the relations $i^2=j^2=-1$ and $ij=-ji$.
Taking $\{1,i,j,ij\}$ as an $\R$-basis of $\mathbb{H}$, we will identify $\mathbb{H}$ with $\R^4$ by expressing each quaternion $x\in \mathbb{H}$ as $x=x_1+x_2i+x_3j +x_4 ij$, and then associating it with the vector $\boldsymbol{x}=(x_1,x_2,x_3,x_4)\in \R^4$. 
The subgroup $\mathbb{H}_1:=\{x\in \mathbb{H}\mid N(x)=1\}$ of $\mathbb{H}^\times$, where $N(x)=x_1^2+x_2^2+x_3^2+x_4^2$ is the \emph{norm} of $x$, is then regarded as the unit sphere $ \mathbb{S}^3$ in $\R^4$.

We will start from the following finite subgroups of $\mathbb{H}_1$ (cf.~\cite[Section 3.5]{ConwaySmith}): the \emph{binary tetrahedral} (resp.~\emph{binary octahedral} and \emph{binary tetrahedral}) group $2T$ (resp.~$2O$ and $2I$) of order 24 (resp.~48 and 120). 
The group $2I$ is also refereed as the icosian group in \cite[p.207]{ConwaySloan}.
The quotient groups $2T/\{\pm1\},2O/\{\pm1\}$ and $2I/\{\pm1\}$ are isomorphic to $A_4,S_4$ and $A_5$, respectively, which are known as exceptional finite subgroups of the rotation group ${\rm SO}(3)$ of $\R^3$ (cf.~\cite[Proposition 11.5.2]{Voight2021}).
As a subset of $\R^4$, each $G\in \{2T,2O,2I\}$ is contained in $ F_G^4$, where $F_{2T}:=\Q,\ F_{2O}:=\Q\big(\sqrt{2}\big)$ and $F_{2I}:=\Q\big(\sqrt{5}\big)$ are number fields.

For $G\in \{2T,2O,2I\}$, let $\mathcal{O}_G:=\langle G\rangle_{R_G}$ be the $R_G$-algebra generated by $G$, where $R_G$ is the ring of integers of the number field $F_G$.
The ring $\mathcal{O}_G$ becomes a maximal $R_G$-order of the quaternion algebra over $F_G$ (cf.~\cite[Chapter 11]{Voight2021}).
In particular, the $\Z$-order $\mathcal{O}_{2T}$ is known as the \emph{Hurwitz order}.
Table \ref{table:orders} below gives a basis of $\mathcal{O}_G$.

\begin{table}[h] 
\caption{Orders}
\label{table:orders}
\begin{tabular}{|l|c|c|c|} \hline
Group $G$ & $2T$ & $2O$ & $2I$\\\hline 
Order $|G|$ & 24 & 48 & 120\\\hline
Number field $F_G$ & $\Q$ & $\Q(\sqrt{2})$ & $\Q(\sqrt{5})$\\\hline
Ring of integers $R_G$ & $\Z$ & $\Z[\sqrt{2}]$ & $\Z[\tau]$\\\hline
$R_G$-order $\mathcal{O}_G$ & $\langle 2T\rangle_\Z$ & $\langle 2O\rangle_{\Z[\sqrt{2}]}$ & $\langle 2I\rangle_{\Z[\tau]}$\\\hline
Basis of $\mathcal{O}_G$ & $\{1, i, j , \omega\}$ & $\{1,\alpha,\beta,\alpha\beta\}$ & $\{1 , i, \zeta, i\zeta\}$\\\hline
\end{tabular}
\end{table}
\noindent
Here we have set
\begin{equation}\label{eq:O_basis_elements}
\begin{aligned}
\tau=\frac{1+\sqrt{5}}{2},\ \omega=\frac{-1+i+j+ij}{2},\ \alpha=\frac{1+i}{\sqrt{2}},\  \beta=\frac{1+j}{\sqrt{2}},\ \zeta=\frac{\tau+\tau^{-1}i+j}{2}.
\end{aligned}
\end{equation}

For a positive integer $m$, we let
\[ \mathcal{O}_{G,m} := \{x\in \mathcal{O}_G\mid \iota_G(N(x))=m\},\]
where $\iota_G:R_G\rightarrow \Z$ is a certain projection defined in \eqref{eq:def_iota}.
A key point is that the composition $\iota_G(N(x))$ defines a positive definite integral quadratic form $\mathcal{Q}_G$ (see the proof of Proposition \ref{prop:counting} for the explicit descriptions).
As such, the theory of modular forms developed by Hecke and Schoeneberg becomes applicable.
Using this framework, we will show that $|\mathcal{O}_{G,m}|>0$ for all $m\in\N$ in Proposition \ref{prop:counting}.
Note that when $G=2T$, the subset $\mathcal{O}_{G,m}\subset \R^4$ for $G= 2T$ coincides with the $m$-shell of the Hurwitz order; that is, the set of lattice points on the sphere with radius $\sqrt{m}$ in $\R^4$.
In contrast, when $G\neq 2T$, the subset $\mathcal{O}_{G,m}\subset \R^4$ does not necessarily lie on a sphere in $\R^4$.

\subsection{Main results}

Let us state our results on Problem \ref{Q1} (i).
For each $G\in\{2T,2O,2I\}$, we will show in Lemma \ref{prop:O_strength} that 
\begin{equation}\label{eq:G_acts_on_O_{G,m}}
T(G)\subset T\big(\mathcal{O}_{G,m}\big), \quad \forall m\in \N.
\end{equation}
This implies that determining the harmonic strength $T(G)$ gives a partial answer to Problem \ref{Q1} (i) for the case $X=\mathcal{O}_{G,m}$.
The first main result of this paper is a complete description of $T(G)$.

\begin{theorem}\label{thm:harmonic_strength}
The harmonic strength of the finite quaternionic groups $2T, 2O$ and $2I$ are given by 
\begin{align*}
&{\rm (i)}\ T(2T) = \{ 10,4,2\}\cup \mathbb{O}^+, \\ 
&{\rm (ii)}\ T(2O) = \{22,14,10,6,4,2 \}\cup \mathbb{O}^+, \\
&{\rm (iii)}\ T(2I) =  \{ 58,46,38,34,28,26,22,18,16,14,10,8,6,4,2\}\cup \mathbb{O}^+,
\end{align*}
where $\mathbb{O}^+$ denotes the set of all positive odd integers.
\end{theorem}

Let us turn to Problem \ref{Q1} (ii).
We first use the \emph{linear programming method} established in \cite{DelsarteGoethalsSeidel77} to obtain a lower bound for the cardinality of finite subsets of $\mathbb{S}^3$ whose harmonic strength includes (more precisely, contains a subset of) $T(G)$.
Our result shows that each $G\in\{2T,2O,2I\}$ is the `minimal' subset of $\mathbb{S}^3$ among those whose harmonic strength contains $T(G)$.  
\begin{theorem}\label{thm:LP_bound}
Let $X$ be an antipodal finite subset of $\mathbb{S}^3$.
\begin{itemize}
\item[(i)] If $ T(X)\supset \{10,4,2\} $, then $|X|\ge 24$.
\item[(ii)] If $  T(X)\supset \{14,10,6,4,2\}$, then $|X|\ge 48$.
\item[(iii)] If $T(X)\supset \{10,8,6,4,2\}$, then $|X|\ge 120$.
\end{itemize}
\end{theorem}

In the case of spherical $t$-designs, the linear programming method provides a so-called Fisher type bound: if $X\subset \mathbb{S}^{d-1}$ is a spherical $t$-design, then $|X|\ge b_{d,t}$, where $b_{d,t}= \binom{d+e-1}{e}+\binom{d+e-2}{e-1}$ if $t=2e$ and $b_{d,t}= 2\binom{d+e-1}{e}$ if $t=2e+1$.
However, it is not known whether the lower bound obtained by this method provides the maximum lower bound of spherical $t$-designs for a given $t$.
Circumstantially, a tight design, meaning that a spherical $t$-design $X$ that attains the lower bound $|X|=b_{d,t}$, is known to have some good extremal properties \cite{BannaiSloane81,CK07}, if it exists, and its classification has therefore been studied extensively.

In a similar vein, the study of classification of `tight' spherical $T$-designs has attracted a lot of attention (although no general formulation of tightness is known).
It was started in \cite{BannaiOkudaTagami15} for the case $T=\{t\}$.
The case $t = 4$ was also investigated in \cite{OkudaYu16}. 
Zhu et al.~\cite{ZhuBannaiBannaiKimYu17} obtained the classification 
of tight spherical $\{t\}$-designs for $t=6,8$, 
as well as the asymptotic non-existence of tight spherical $\{2e\}$-designs for $e \geq 3$. 
They also studied the existence problem for tight spherical $T$-designs for some $T$, including the case $T = \{8, 4\}$. 

Our classification problem is based on the fact that, for an antipodal spherical $T$-design $X$ of $\mathbb{S}^{d-1}$ with $N$ points, 
its orthogonal transformation $g(X):=\{g(\boldsymbol{x})\mid \boldsymbol{x}\in X\}$ for $g\in O(\R^d)$ is again an antipodal spherical $T$-design of $\mathbb{S}^{d-1}$ with $N$ points, where $O(\R^d):=\{g:\R^d\rightarrow \R^d:\mbox{linear map}\mid\langle g (\boldsymbol{x}),g(\boldsymbol{y})\rangle=\langle\boldsymbol{x},\boldsymbol{y}\rangle, \ \forall \boldsymbol{x},\boldsymbol{y}\in \R^d\}$ is the orthogonal transformation group. 
We say $X$ is {\itshape unique} as an antipodal spherical $T$-design of $\mathbb{S}^{d-1}$ with $N$ points if every antipodal spherical $T$-design of $\mathbb{S}^{d-1}$ with $N$ points is an orthogonal transformation of $X$.

\begin{theorem}\label{thm:uniqueness}
\begin{itemize}
\item[(i)] The group $2T$ is unique as an antipodal spherical $\{10,4,2\}$-design of $\mathbb{S}^3$ with $24$ points.
\item[(ii)] The group $2O$ is unique as an antipodal spherical $\{14,10,6,4,2\}$-design of $\mathbb{S}^3$ with $48$ points.
\item[(iii)] The group $2I$ is unique as an antipodal spherical $\{10,8,6,4,2\}$-design of $\mathbb{S}^3$ with $120$ points.
\end{itemize}
\end{theorem}

Theorem \ref{thm:uniqueness} is referred to as a uniqueness theorem in the study of the classification of spherical designs.
Remark that Theorem \ref{thm:uniqueness} (i) is obtained from Theorem 1.1 in our previous work \cite{HiraoNozakiTasaka}; we have shown the uniqueness of the normalized $D_4$ root system $\frac{1}{\sqrt{2}}{\bf D}_4$, all permutations of $\frac{1}{\sqrt{2}}(\pm1,\pm1,0,0)$, as an antipodal spherical $\{10,4,2\}$-design of $\mathbb{S}^3$ with 24 points. 
Indeed, the set $2T$ is an orthogonal transformation of $\frac{1}{\sqrt{2}}{\bf D}_4$ (see Remark \ref{rem:D_4}).
Theorem \ref{thm:uniqueness} (iii) is also a consequence of 
the uniqueness of the 600-cell as a spherical $11$-design of $\mathbb{S}^3$ with 120 points shown in \cite{BoyvalenkovDanev01}.
In this paper, we will prove Theorem \ref{thm:uniqueness} (ii) only.

\subsection{Spherical theta functions}
Let $G\in\{2T,2O,2I\}$. 
We will also consider a connection with the \emph{spherical theta function} $\theta_{G,P}(z) $ defined for $P\in {\rm Harm}_\ell(\R^4)$ by
\begin{equation}\label{eq:theta} 
\theta_{G,P}(z) := \sum_{x\in \mathcal{O}_{G}} P(x) q^{\iota_G(N(x))} \quad (q=e^{2\pi iz}).
\end{equation}
This function is holomorphic on the complex upper half-plane and becomes an elliptic modular form of weight $2+\ell$ for $\Gamma_0(2)$ when $G=2T$, of weight $4+\ell$ for $\Gamma_0(2)$ when $G=2O$, and of weight $4+\ell$ for ${\rm SL}_2(\Z)$ when $G=2I$, as shown by Hecke and Schoeneberg (cf.~\cite[\S3.1]{Ebeling}, \cite[\S4.9]{Miyake}, \cite[\S6]{Ogg}).
The spherical theta function can be written as
\begin{equation*}\label{eq:q_expansion}
\theta_{G,P}(z) = \sum_{m\ge0} \bigg(\sum_{x\in\mathcal{O}_{G,m}} P(x)\bigg)q^{m}.
\end{equation*} 
From this, we see that, for $\ell\in \N$, the condition $\ell \in T\big(\mathcal{O}_{G,m}\big)$ is equivalent to the statement that the coefficient of $q^{m}$ vanishes in each of modular forms $\theta_1,\ldots,\theta_v$, where $\theta_1,\ldots,\theta_v$ form a basis of the $\C$-vector space 
\begin{equation*}\label{eq:def_V} 
\Theta(G,\ell):= \langle \theta_{G,P}(z) \mid P\in {\rm Harm}_\ell(\R^4)\rangle_\C,
\end{equation*}
which is a subspace of the $\C$-vector space of modular forms.
In particular, if $\dim_\C \Theta(G,\ell)=0$, then $\ell \in T(\mathcal{O}_{G,m})$ holds for all $m\in \N$.
This idea was employed in \cite{Pache05, Venkov84} to determine which degrees are included in the harmonic strength of the shells of a lattice. In fact, by applying the theory of modular forms to estimate the dimension of the space of spherical theta functions of a lattice, they partially determined the harmonic strength of the shells of several important lattices.

In our case, we obtain the following relationship between the harmonic strength $T(G)$ and the dimension of the space $\Theta(G,\ell)$.

\begin{theorem}\label{thm:dim=0}
For $G\in\{2T,2O,2I\}$, we have that $\dim_\C \Theta(G,\ell)=0$ if and only if $\ell \in T(G)$.
\end{theorem}

It is worth emphasizing that, combining Theorems \ref{thm:harmonic_strength} and \ref{thm:dim=0}, we can determine all $\ell\in \N$ for which the space $\Theta(G,\ell)$ is zero-dimensional.
Our proof does not rely on the theory of modular forms.

We remark that the problem whether the space of modular forms is generated by spherical theta functions is a version of the basis problem as raised by Eichler \cite{Eichler} (see also \cite{Martin} and references therein). 
While they use all classes of lattices with the same genus, we focus on the spherical theta functions associated with a single lattice, from the viewpoint of spherical designs.

\subsection{Contents}
In Section 2, we give explicit descriptions of the groups $2T$, $2O$, and $2I$.
We then recall the notion of spherical designs and present a method, using Gegenbauer polynomials and distance distributions, to determine whether $\ell \in T(X)$ for a given finite subset $X \subset \mathbb{S}^{d-1}$.
Section 3 is devoted to the proofs of our main results: Theorems \ref{thm:harmonic_strength}, \ref{thm:LP_bound}, and \ref{thm:uniqueness}.
In Section 4, we first describe the quadratic form $\mathcal{Q}_G$ obtained from $\iota_G(N(x))$ and compute the set $\mathcal{O}_{G,1}$ to prove that $|\mathcal{O}_{G,m}| > 0$ for all $m \in \mathbb{N}$, using the theory of modular forms.
Theorem \ref{thm:dim=0} is proved based on the computation of $\mathcal{O}_{G,1}$.
Additional data on $\Theta(G, \ell)$ such as a dimension conjecture will be given.

\section{Preliminaries}
\subsection{Exceptional quaternionic unit groups}
We provide a brief overview of quaternions from the books \cite{ConwaySmith,Coxeter}.
Let $\mathbb{H}=\{x_1+x_2i+x_3j +x_4k \mid x_1,\ldots,x_4\in \R\}$ be the $\R$-algebra of quaternions, where $1,i,j,k$ are an $\R$-basis such that $i^2=j^2=-1$ and $ij=k=-ji$.
Its (non-commutative) multiplication is given for $x=x_1+x_2i+x_3j+x_4k$ and $y = y_1+y_2i +y_3j+ y_4k$ by
\begin{equation}\label{eq:product}
\begin{aligned}
xy&= (x_1 y_1 - x_2 y_2 - x_3 y_3 - x_4 y_4)+( x_2 y_1 + x_1 y_2 - x_4 y_3 + x_3 y_4)i\\
&+( x_3 y_1 + x_4 y_2 + x_1 y_3 - x_2 y_4)j+ ( x_4 y_1 - x_3 y_2 + x_2 y_3 + x_1 y_4)k.
\end{aligned}
\end{equation}
For $x=x_1+x_2i+x_3j+x_4k\in \mathbb{H}$, denote by $\bar{x}=x_1-x_2i-x_3j- x_4k$ the \emph{conjugate} of $x$ and by 
$N(x)=x\bar{x}=
x_1^2+x_2^2+x_3^2+x_4^2$
the (reduced) \emph{norm} of $x$.
The norm map $N: \mathbb{H}^\times\rightarrow \R_{>0}^\times$ is a surjective group homomorphism.

The set of unit quaternions $\mathbb{H}_1=\{\varepsilon \in \mathbb{H}\mid N(\varepsilon)=1\}$ forms a subgroup of $\mathbb{H}^\times$.
Let $\varepsilon \in \mathbb{H}_1$. 
We note that $\varepsilon^{-1}=\bar{\varepsilon}$.
For $x\in {\rm Im}\ \mathbb{H}:=\{x_2i+x_3j+x_4k\mid x_2,x_3,x_4\in \R\}$, we see that $\varepsilon x \varepsilon^{-1} \in {\rm Im}\, \mathbb{H}$.
The matrix $Z_x$ defined by the coefficient matrix $(\varepsilon i\varepsilon^{-1},\varepsilon j \varepsilon^{-1},\varepsilon k \varepsilon^{-1})=(i,j,k) Z_\varepsilon$ is a rotation matrix, i.e., $Z_x\in {\rm SO}(3):=\{A\in {\rm O}(3)\mid  \det A=1\}$.
From this, the map
\[ \rho:  \mathbb{H}_1 \rightarrow {\rm SO}(3), \quad \varepsilon \mapsto Z_\varepsilon\]
is a homomorphism.
Moreover, the map $\rho$ is surjective.
Its kernel is $\{\pm 1\}$ (\cite[Corollary 2.4.21]{Voight2021}).
Namely, we obtain
\[\mathbb{H}_1/\{\pm1\}\cong {\rm SO}(3).\]
Thus, for any finite subgroup $G$ of $\mathbb{H}_1$ with $-1\in G$, we have the embedding $G/\{\pm 1\}\rightarrow {\rm SO}(3)$.
Here we recall that every finite subgroups of ${\rm SO}(3)$ is completely classified as follows (cf.~\cite[Proposition 11.5.2]{Voight2021}): 
\begin{itemize}
\item[(i)] a cyclic group; 
\item[(ii)] a dihedral group;
\item[(iii)] the tetrahedral group $A_4$ of order 12;
\item[(iv)] the octahedral group $S_4$ of order 24; or
\item[(v)] the icosahedral group $A_5$ of order 60.
\end{itemize}

Finite subgroups of $\mathbb{H}_1$ whose images under the map $\rho$ are isomorphic to the cases (iii), (iv) and (v) above are called the binary tetrahedral group $2T$ of order 24, the binary octahedral group $2O$ of order 48 and the binary icosahedral group $2I$ (of order 120), respectively.
Their generators can be found in \cite[Chapter 6.5]{Coxeter}.
We follow the definition given in \cite[Section 3.5]{ConwaySmith}.

The representations of the groups $2T$, $2O$, and $2I$ we use are as follows.
Denote by $Q_8 = \{\pm1,\pm i,\pm j,\pm k\}$ the \emph{quaternion group} of order 8.
For a subgroup $G\subset \mathbb{H}^\times$, we write 
\[xG=\{x\varepsilon \mid \varepsilon \in G\}\]
for the $G$-orbit of $x\in \mathbb{H}$.
Note that $|xG|=|G|$ for $x\in \mathbb{H}^\times $.
With these notations, the exceptional quaternionic unit groups $2T$, $2O$ and $2I$ can be written as follows:
\begin{equation}\label{eq:def_exceptional_groups}
\begin{aligned}
2T&=Q_8 \cup \omega\,Q_8 \cup \omega^2\,Q_8,\\
2O&=2T \cup \alpha\, 2T,\\
2I&=2T \cup \zeta\, 2T\cup \zeta^2\, 2T \cup \zeta^3\, 2T \cup \zeta^4\, 2T,
\end{aligned}
\end{equation}
where $\omega,\alpha,\zeta$ are defined in \eqref{eq:O_basis_elements}.
Note that these elements are units: $\omega^3=1$ and $\alpha^4=\zeta^5=-1$.
The description given in \eqref{eq:def_exceptional_groups} coincides with that in \cite[Section 3.5]{ConwaySmith}, which can be verified using symbolic computations performed with the help of the computer software Mathematica.

For latter purpose, we recall that the multiplication by a quaternion yields an orthogonal transformation.

\begin{lemma}\label{lem:quaternion_orthogonal}
For $x\in \mathbb{H}^\times$, the $\R$-linear map 
\[ m_x:\R^4\rightarrow \R^4, \ y\mapsto \frac{x}{\sqrt{N(x)}} y\] 
is an orthogonal transformation, i.e., $m_x\in O(\R^4)$.
Here we are identifying $\mathbb{H}$ with $\R^4$ as $\R$-vector spaces by taking $(y_1,y_2,y_3,y_4)\in \R^4$ for $y=y_1+y_2i+y_3j+y_4k\in \mathbb{H}$. 
\end{lemma}
\begin{proof}
For $x=x_1+x_2i+x_3j+x_4k\in \mathbb{H}^\times$, let
\begin{equation}\label{eq:M_x}
M_x:=\frac{1}{\sqrt{N(x)}} \begin{pmatrix}x_1&x_2&x_3&x_4 \\ -x_2&x_1&x_4 &-x_3\\ -x_3&-x_4 & x_1 & x_2\\ -x_4 & x_3 & -x_2 & x_1\end{pmatrix}.
\end{equation}
From \eqref{eq:product} we obtain
\[\frac{1}{\sqrt{N(x)}}xy=yM_x ,\quad (\forall y\in \mathbb{H}),\]
where $yM_x$ is the matrix product under the identification $\mathbb{H}$ with $\R^4$.
Since $M_x$ is an orthogonal matrix, we obtain the desired result.
\end{proof}

Note that the map
\[\mathbb{H}_1 \longrightarrow O(4),\ x\longmapsto M_x\]
is an injective group homomorphism, where $O(4)$ is the orthogonal group.

\begin{remark}\label{rem:D_4}
In our previous paper \cite{HiraoNozakiTasaka}, we have defined the $D_4$ root system ${\bf D}_4$ as the set of all permutations of $(\pm1,\pm1,0,0)$.
Using the identification of $\mathbb{H}$ with $\R^4$, one can check that $\frac{1}{\sqrt{2}}{\bf D}_4=\alpha 2T=m_\alpha(2T)$, where $\alpha$ is given in \eqref{eq:O_basis_elements}.
From Lemma \ref{lem:quaternion_orthogonal}, we see that the set $\frac{1}{\sqrt{2}}{\bf D}_4$ is an orthogonal transformation of $2T$.
\end{remark}

\subsection{Spherical design}
The concept of spherical $T$-designs introduced by Delsarte and Seidel \cite{DS89} applies to finite subsets of the unit sphere ${\mathbb S}^{d-1}: = \{ \boldsymbol{x} = (x_1, \ldots, x_d) \in \R^{d} \mid \langle \boldsymbol{x}, \boldsymbol{x} \rangle = 1 \}$ in the $d$-dimensional Euclidean space $\R^d$, where for $\boldsymbol{x},\boldsymbol{y}\in \R^d$, its inner product is denoted by
\[ \langle \boldsymbol{x}, \boldsymbol{y} \rangle := \sum_{a = 1}^{d} x_ay_a.\]
Under the identification $\mathbb{H}$ with $\R^4$, for $x\in \mathbb{H}$ we have $N(x)=\langle \boldsymbol{x}, \boldsymbol{x} \rangle $, and hence, $\mathbb{S}^3=\mathbb{H}_1$.

For an integer $\ell\ge0$, denote by ${\rm Hom}_{\ell}(\R^d)$ the $\R$-vector subspace of $\R[x_1, \ldots, x_d]$ consisting of all homogeneous polynomials of degree $\ell$ in $d$ variables.
We write
\[\Delta:= \frac{\partial^2}{\partial x_1^2}+\cdots +\frac{\partial^2}{\partial x_d^2}\]
for the Laplacian operator.
Let 
\[{\rm Harm}_{\ell}(\R^d) :=\{P\in {\rm Hom}_{\ell}(\R^d)\mid \Delta P=0\}\]
be the $\R$-vector space of harmonic polynomials of homogeneous degree $\ell$.
It is well known (see e.g., Theorem 3.2 in \cite{DelsarteGoethalsSeidel77}) that $\dim_\R {\rm Harm}_\ell(\R^d)= \binom{\ell+d-1 }{\ell} - \binom{\ell+d-3}{\ell -2}$ holds for $\ell,d\in \N$ with $\ell\ge2$.

\begin{definition}\label{def:T-design}
For a subset $T$ of $\N$, a non-empty finite subset $X$ of $\mathbb{S}^{d-1}$ is called a {\itshape spherical $T$-design}
if for any $\ell \in T$ and $P \in {\rm Harm}_\ell (\R^d)$ we have that
\[\sum_{\boldsymbol{x} \in X} P(\boldsymbol{x}) = 0.\]
\end{definition}

Alternatively, a finite subset $X\subset \mathbb{S}^{d-1}$ is a spherical $T$-design if $T\subset T(X)$, where $T(X)$ is the harmonic strength of $X$ defined in \eqref{eq:T(X)}.

For a subset $X$ of $\R^d$, we write $-X:=\{-\boldsymbol{x}\mid \boldsymbol{x}\in X\}$. 
A set $X$ is said to be \emph{antipodal} if we have $-X=X$. 
If a finite subset $X\subset \mathbb{S}^{d-1}$ is antipodal, then the harmonic strength $T(X)$ contains all positive odd integers.
Conversely, for a finite subset $X\subset \mathbb{S}^{d-1}$, if $T(X)$ contains all positive odd integers, then $X$ is antipodal (cf.~\cite{MisawaMunemasaSawa}).

For an antipodal finite subset $X$ of $\R^d$, a subset $X'\subset X$ is called \emph{a half set of $X$} if $X$ is a disjoint union of $X'$ and $-X'$; $X'\sqcup (-X')=X$.
It follows that
\begin{equation}\label{eq:antipodal_cardinality}
|X|=2|X'|.
\end{equation}
For any antipodal finite subset $X$ of $\mathbb{S}^{d-1}$ (note that $\boldsymbol{0}\not\in X$), a half set of $X$ always exists, but not unique.
The harmonic strength of the half set satisfies $T(X') \cap 2\Z = \{2\ell \in \N\mid 2\ell\in T(X)\}$ (see e.g. \cite[Lemma 2.2]{HiraoNozakiTasaka}).
It should be noted that $T(X')$ may contain odd integers (see \cite{BannaiZhaoZhuZhu18}).

\subsection{Computing the harmonic strength}
We now explain our method to compute the harmonic strength defined in \eqref{eq:T(X)}.
We use the Gegenbauer (or ultraspherical) polynomials which form a class of orthogonal polynomials (cf.~\cite[p.302]{AndrewsAskeyRoy} and \cite[(4.7.23)]{S75}).
For $\lambda>0$, the Gegenbauer polynomial $C_\ell^\lambda(s)\in \R[s]$ of degree $\ell$ is defined by 
the coefficient of $u^\ell$ in the formal power series $\Phi^{\lambda}(s;u):=(1-2su+u^2)^{-\lambda}$:
\[
\Phi^{\lambda}(s;u)=\sum_{\ell \geq 0}C_\ell^\lambda(s)u^\ell.
\]
It is also obtained by the recurrence relation $\ell C_\ell^\lambda(s) = 2(\ell+\lambda-1)s C_{\ell-1}^\lambda(s) -(\ell+2\lambda-2)C_{\ell-2}^\lambda(s)$ for $\ell\ge2$ with $C_0^\lambda(s)=1,\ C_1^\lambda(s)=2\lambda s$.
For example, the case $\ell=2$ is $C_2^{\lambda}(s)=-\lambda +2\lambda (1+\lambda) s^2$.
It follows that $C_\ell^\lambda(1) = \frac{2\lambda (2\lambda+1)(2\lambda+2)\cdots (2\lambda+\ell-1)}{\ell!}$.

\begin{lemma}\label{lem:gegenbauer}
Let $X$ be a finite subset of $\mathbb{S}^{d-1}$.
For $\ell\in \N$, we have that
\[\sum_{\boldsymbol{x},\boldsymbol{y}\in X} C_{\ell}^{(d-2)/2}(\langle \boldsymbol{x},\boldsymbol{y}\rangle ) =0\]
if and only if $\ell\in T(X)$, where $T(X)$ is the harmonic strength of $X$ defined in \eqref{eq:T(X)}.
\end{lemma}
\begin{proof}
Write $c_{\ell,d}:=\dim_\R {\rm Harm}_\ell(\R^d)$.
Let $\{S_{\ell,n}^{(d)}\mid n=1,2,\ldots,c_{\ell,d} \}$ be an orthonormal basis of ${\rm Harm}_\ell(\R^d)$ with respect to the inner product defined for 
$F_1,F_2\in {\rm Harm}_\ell(\R^d)$ by 
\[( F_1,F_2) :=\frac{1}{|\mathbb{S}^{d-1}|} \int_{\boldsymbol{x}\in \mathbb{S}^{d-1}} F_1(\boldsymbol{x})F_2(\boldsymbol{x})d\sigma(\boldsymbol{x}).\]

From the addition theorem \cite[Theorem 9.6.3]{AndrewsAskeyRoy}, for $\ell\in \N$ and $\boldsymbol{x},\boldsymbol{y}\in \mathbb{S}^{d-1}$, one has
\[ \sum_{n=1}^{c_{\ell,d}}S_{\ell,n}^{(d)}(\boldsymbol{x})S_{\ell,n}^{(d)}(\boldsymbol{y}) = \frac{c_{\ell,d}}{|\mathbb{S}^{d-1}|} C_\ell^{(d-2)/2}(\langle \boldsymbol{x},\boldsymbol{y}\rangle).\]
Let $Z_n := \sum_{\boldsymbol{x}\in X}S_{\ell,n}^{(d)}(\boldsymbol{x})\in \R$.
Then we have
\[\sum_{n=1}^{c_{\ell,d}}Z_n^2= \sum_{\boldsymbol{x},\boldsymbol{y}\in X}\sum_{n=1}^{c_{\ell,d}}S_{\ell,n}^{(d)}(\boldsymbol{x})S_{\ell,n}^{(d)}(\boldsymbol{y}) = \frac{c_{\ell,d}}{|\mathbb{S}^{d-1}|}  \sum_{\boldsymbol{x},\boldsymbol{y}\in X} C_{\ell}^{(d-2)/2}(\langle \boldsymbol{x},\boldsymbol{y}\rangle )  .\]
Thus, the result follows from the definition: $\ell \in T(X)$ if and only if $Z_n=0$ holds for all $n=1,2,\ldots, c_{\ell,d}$.
\end{proof}

\begin{proposition} \label{prop:gene}
Let $d\geq 3$.
For a finite subset $X$ of $\mathbb{S}^{d-1}$ and $\ell\in \N$, we have that $\ell \in T(X)$ if and only if 
the coefficient of $u^\ell$ in $\sum_{\boldsymbol{x},\boldsymbol{y} \in X}\Phi^{(d-2)/2}(\langle \boldsymbol{x},\boldsymbol{y} \rangle;u)$ is zero. 
\end{proposition}
\begin{proof}
    Since 
\begin{align*}
\sum_{\boldsymbol{x},\boldsymbol{y} \in X}\Phi^{(d-2)/2}(\langle \boldsymbol{x},\boldsymbol{y} \rangle;u)
&=
\sum_{\boldsymbol{x},\boldsymbol{y} \in X} \sum_{\ell\geq 0} C_\ell^{(d-2)/2}(\langle \boldsymbol{x},\boldsymbol{y}  \rangle) u^\ell   \\
&=
 \sum_{\ell\geq 0} \left(\sum_{\boldsymbol{x},\boldsymbol{y} \in X}C_\ell^{(d-2)/2}(\langle \boldsymbol{x}, \boldsymbol{y} \rangle)\right) u^\ell, 
\end{align*}
the result follows from Lemma \ref{lem:gegenbauer}.
\end{proof}

The power series $\sum_{\boldsymbol{x},\boldsymbol{y} \in X}\Phi^{(d-2)/2}(\langle \boldsymbol{x},\boldsymbol{y} \rangle;u)\in \R[[u]]$ as in Proposition \ref{prop:gene} can be computed from the \emph{distance distribution} $R_X:=\{(s, A_s(X))\mid s \in A(X)\cup\{1\}\}$ of $X$, where for $s\in \R$, we write
\begin{equation}\label{eq:d_a}
A_s(X):=|\{(\boldsymbol{x},\boldsymbol{y})\in X\times X\mid \langle \boldsymbol{x},\boldsymbol{y}\rangle=s\}|
\end{equation}
and 
\begin{equation}\label{eq:A(X)}
A(X):=\{\langle \boldsymbol{x},\boldsymbol{y}\rangle \mid \boldsymbol{x},\boldsymbol{y}\in X, \ \boldsymbol{x}\neq\boldsymbol{y}\}
\end{equation}
denotes the set of inner products of two distinct elements in $X\subset \mathbb{S}^{d-1}$.
Note that $A_1(X)=|X|$.
With this, one has
\begin{equation*}\label{eq:Phi_A(X)}
\sum_{\boldsymbol{x},\boldsymbol{y} \in X} \Phi^{(d-2)/2}(\langle \boldsymbol{x},\boldsymbol{y} \rangle;u)=\sum_{s\in A(X)\cup\{1\}}A_s(X) \Phi^{(d-2)/2}(s;u).
\end{equation*}
Our strategy is to express the above right-hand side in terms of simple rational functions in $u$ and check the conditions under which the coefficient of $u^\ell$ vanishes.

Below, we only use the case $d=4$. In this case, we simply write
\[\Phi(s;u):=\Phi^{1}(s;u)=\frac{1}{1-2su+u^2}.\]

\section{Proofs of main results}
\subsection{Proof of Theorem \ref{thm:harmonic_strength}}
To prove Theorem \ref{thm:harmonic_strength}, we use the theory of invariant polynomials.
Let ${\rm Hom}_{\ell}(\C^2)$ denote the $\C$-vector subspace of $\C[z_1, z_2]$ consisting of all homogeneous polynomials of degree $\ell$ in two variables.
Let us recall an action of the group $\mathbb{H}_1$ on ${\rm Hom}_\ell (\C^2)$.

We first note that a quaternion $x \in \mathbb{H}$ can be written as $x=z_1+z_2j $ for some $z_1,z_2\in \C$:
\[x = x_1 + x_2 i + x_3 j + x_4 k = (x_1 + x_2 i) + (x_3 + x_4 i) j.\]
For $x=z_1+z_2j\in \mathbb{H}$, we set 
\[ C_x := \begin{pmatrix} z_1&z_2\\ -\overline{z_2}&\overline{z_1}\end{pmatrix},\]
where $\overline{z}$ denotes the complex conjugate for $z\in \C$.
The matrix $C_x$ lies in the special unitary group $SU(2)=\{A\in U(2)\mid \det A=1\}$ if $x\in \mathbb{H}_1$.
Moreover, the map
\begin{equation}\label{eq:map_C_x}
\mathbb{H}_1\longrightarrow SU(2),\ \varepsilon\longmapsto C_{\varepsilon}
\end{equation}
is an injective group homomorphism.
This shows that for $\varepsilon\in \mathbb{H}_1$ and $ P\in {\rm Hom}_{\ell}(\C^2)$, setting
\[ (\varepsilon P)(z_1,z_2):=P\big( (z_1,z_2)C_\varepsilon\big)\]
we obtain an action of $\mathbb{H}_1$ on the space ${\rm Hom}_{\ell}(\C^2)$.

Let $G$ be a finite subgroup of $\mathbb{H}_1$, which can be viewed as a subgroup of $SU(2)$ via \eqref{eq:map_C_x}.
Denote by ${\rm Hom}_{\ell}(\C^2)^G=\{P\in {\rm Hom}_{\ell}(\C^2)\mid  \varepsilon P=P, \ \forall \varepsilon \in G\}$ the $\C$-vector space of $G$-invariant homogeneous polynomials, and by
\[\Psi_G(u):=\sum_{\ell\geq 0} \dim_\C {\rm Hom}_\ell (\C^2)^G u^\ell\]
the \emph{Molien series} (or Hilbert--Poincar\'{e} series) of $G$.
It is well known (cf.~\cite[Theorem 4.3.2]{Smith}) that the Molien series $\Psi_G(u)$ can be calculated by
\[  \Psi_G(u)= \frac{1}{|G|} \sum_{\varepsilon \in G} \frac{1}{\det (I-uC_\varepsilon)}.\]

We now prove that, for a finite subgroup $G$ of $\mathbb{H}_1$, the formal power series $(1/|G|^2)\sum_{x,y \in G} \Phi(\langle x,y \rangle;u)$ coincides with the Molien series $\Psi_G(u)$ of $G$. 

\begin{lemma}\label{lem:SU(2)}
Let $G$ be a finite subgroup of $\mathbb{H}_1$. Then we have  
\[\Psi_G(u)= \frac{1}{|G|^2}\sum_{x,y \in G} \Phi(\langle x,y \rangle; u),\]
where the inner product $\langle x,y \rangle$ is given by $\langle x_1+x_2 i+x_3 j +x_4 k,y_1+y_2i+y_3 j+y_4 k \rangle=x_1y_1+x_2y_2+x_3y_3+x_4y_4$ via the identification $\mathbb{H}$ with $\R^4$.
\end{lemma}
\begin{proof}
From Lemma \ref{lem:quaternion_orthogonal}, for $x,y\in \mathbb{H}_1$, we have $\langle x,y\rangle = \langle 1,x^{-1}y\rangle$.
Thus, it follows that 
\begin{align*}
\frac{1}{|G|^2}\sum_{x,y \in G} \Phi(\langle x,y \rangle; u)&=
\frac{1}{|G|}\sum_{x \in G} \frac{1}{1-2\langle 1,x \rangle u+u^2}    \\
&=\frac{1}{|G|}\sum_{x\in G} \frac{1}{1-2x_1 u+u^2} \\
&= \frac{1}{|G|}\sum_{x \in G}\frac{1}{\det (I-u C_x)},
\end{align*}
where $\det (I-u C_x)=1-2x_1 u+u^2$ is easily verified. 
\end{proof}

For a finite subgroup $G\subset \mathbb{H}_1$ viewed as a finite subgroup of $SU(2)$, the Molien series $\Psi_G(u)$ is explicitly calculated as follows \cite[Section 4.5]{Sbook}: 
\begin{table}[h]
    \centering
    \caption{Molien series of $G$}
    \label{tab:molien}
    \begin{tabular}{|c|c|} \hline
    $G$ & $\Psi_G(u)$
    \\ \hline \hline 
       $C_n$  &  $\displaystyle{\frac{1-u^{2n}}{(1-u^2)(1-u^n)^2}=\frac{1+u^{n}}{(1-u^2)(1-u^n)}}$\\ \hline 
       $D_{2n}$ &$\displaystyle{\frac{1-u^{4n+4}}{(1-u^4)(1-u^{2n})(1-u^{2n+2})}=\frac{1+u^{2n+2}}{(1-u^4)(1-u^{2n})}}$  \\ \hline
       $2T$ & $\displaystyle{\frac{1-u^{24}}{(1-u^6)(1-u^8)(1-u^{12})}=\frac{1+u^{12}}{(1-u^6)(1-u^8)}}$\\ \hline 
       $2O$ & $\displaystyle{\frac{1-u^{36}}{(1-u^8)(1-u^{12})(1-u^{18})}=\frac{1+u^{18}}{(1-u^8)(1-u^{12})}}$\\ \hline 
       $2I$ & $\displaystyle{\frac{1-u^{60}}{(1-u^{12})(1-u^{20})(1-u^{30})}=\frac{1+u^{30}}{(1-u^{12})(1-u^{20})}}$\\ \hline 
    \end{tabular}
\end{table}

\noindent
Here $C_n$ and $D_{2n}$ are the cyclic and dihedral group of order $n$ and $4n$, respectively.
A quaternionic representation is given as follows:
\[C_{n}:=\{e^{2m\pi i/n}\mid m=0,1,\ldots, n-1\}, \quad D_{2n}:=C_{2n} \cup C_{2n}',\]
where we set $e^{i\theta}=\cos \theta + i \sin \theta$ and $C_{2n}':=\{e^{m\pi i/n}j \mid m=0,1\ldots, 2n -1\}$. 

\begin{proof}[Proof of Theorem \ref{thm:harmonic_strength}]
       We compute the generating series in Proposition \ref{prop:gene}.
       As a prototype, we give the proof for the case $X=2T$ (Theorem \ref{thm:harmonic_strength} (i)); the proofs for the other cases are similar.

    From Lemma \ref{lem:SU(2)} and Table \ref{tab:molien}, one has
    \begin{align*}
        \frac{1}{|2T|^2}\sum_{x,y \in 2T} \Phi(\langle x, y \rangle;u)&=\frac{1+u^{12}}{(1-u^6) (1-u^8)}.
    \end{align*}
    
    We now determine where the coefficient of $u^\ell$ in the above power series is zero.
    Let $B_{\ell}=\{(a,b)\in \Z_{\ge0}^2\mid 3a+4b=\ell\}$.
    Then we have
    \[ \frac{1}{(1-u^6)(1-u^8)} = \sum_{a,b\ge0}u^{6a+8b}=\sum_{\ell\ge0}|B_\ell|u^{2\ell}.\]
    We will show that $|B_\ell| =0$ if and only if $\ell=1,2,5$, 
    from which by Proposition \ref{prop:gene} the result 
    $T(2T) = \{ 10,4,2\}\cup \mathbb{O}^{+}$ 
    follows.
    For $\ell\in \N$, suppose $|B_\ell|>0$. Then there exists $(a,b)\in  \Z_{\ge0}^2$ such that $\ell= 3a+4b=3(a+b)+b$.
    Therefore, $|B_\ell|>0$ holds if and only if $\ell \in S_0\cup S_1\cup S_2$ holds, where, for $b\in\{0,1,2\}$, we let $S_b=\{3a+b\in \Z \mid a\ge b \}$.
    Since $\N\setminus \big(S_0\cup S_1\cup S_2\big)=\{1,2,5\}$, we have done.

 By the same argument as for $X=2T$, we can find all degree $\ell$ at which the coefficient of $u^\ell$ becomes zero in the power series given in Table \ref{tab:molien}.
    The detailed verification is left to the reader. 
\end{proof}

The harmonic strengths of the groups $C_n$ and $D_{2n}$ can also be obtained from Lemma~\ref{lem:SU(2)} and Table~\ref{tab:molien}.
We state the results without providing proofs.

\begin{theorem}\label{thm:harmonic_strength_Dih}
The harmonic strengths of the cyclic group $C_n$ and the dihedral group $D_{2n}$ are given by
\[
T(C_n)=\mathbb{O}^+ \qquad (n\geq 2),  
\]
and 
\[T(D_{2n}) =\{2\ell \in \N \mid 0<\ell < n, \ \ell \equiv 1\bmod 2  \} \cup 
\mathbb{O}^{+}\qquad  (n\geq 2). 
\]
\end{theorem}

For subgroups of $\mathbb{H}_1$, the formal power series $\sum_{x,y \in X} \Phi(\langle x,y \rangle; u)$ can be written explicitly as shown in Lemma~\ref{lem:SU(2)}.
In general, the crucial point in this approach is that the series $\sum_{x,y \in X} \Phi(\langle x,y \rangle; u)$ can be rewritten in a sufficiently simple form to determine the non-vanishing of its coefficients.
However, finding such a simple form is generally difficult, as the coefficients of this series may not be integers.

\subsection{Proof of Theorem \ref{thm:LP_bound}}
In order to give a lower bound on the cardinality of a spherical $T$-design, we recall a useful inequality due to Delsarte--Goethals--Seidel \cite{DelsarteGoethalsSeidel77}.

Let 
\[ Q_\ell^{(d)}(s):=\frac{d+2\ell-2}{d-2} C_{\ell}^{(d-2)/2}(s)\]
be the scaled Gegenbauer polynomial used in \cite{DelsarteGoethalsSeidel77}.
With this normalization, we have $Q_\ell^{(d)}(1)=\dim_\R {\rm Harm}_\ell(\R^d)$.
As examples, one has
\begin{equation}\label{eq:Q_l^{d}}
\begin{aligned}
Q_2^{(d)}(s) &=\frac{1}{2} (d+2) \left(d s^2-1\right),\quad Q_3^{(d)}(s) =\frac{1}{6} d (d+4) s \left((d+2) s^2-3\right),\\
Q_4^{(4)}(s) &=5\left( 16 s^4-12 s^2+1\right),\quad Q_6^{(4)}(s) =7 \left(64 s^6-80 s^4+24 s^2-1\right).
\end{aligned}
\end{equation}

Hereafter, we fix $d=4$.
Recall that the polynomials $Q_\ell(s)=Q_\ell^{(4)}(s)$ form a basis of the polynomial algebra $\R[s]$, and every polynomial $F(s)\in \R[s]$ of degree $r$ can be written uniquely in the form
\begin{equation}
\label{eq:F}
F(s) = \sum_{\ell = 0}^{r} f_\ell Q_{\ell} (s),
\end{equation}
called the \emph{Gegenbauer expansion} of $F(s)$.
Note that, since the polynomials $Q_\ell(s)$ are the orthogonal polynomials on the closed interval $[-1, 1]$ with respect to the inner product of the weight function $(1 - s^2)^{1/2}$, the coefficient $f_\ell\in \R$ given in \eqref{eq:F} can be computed from
\[
f_\ell = \frac{ \int_{-1}^{1} F(s) Q_{ \ell} (s) (1 - s^2)^{\frac12} \; ds}{ \int_{-1}^{1} Q_\ell(s)^2(1 - s^2)^{\frac12} \; ds}.
\]

\begin{lemma}\label{lem:LP_bound}
Let $X\subset \mathbb{S}^3$ be a spherical $T$-design and $F(s)\in \R[s]$ of degree $r$.
Suppose that the Gegenbauer expansion \eqref{eq:F} of $F(s)$ satisfies 
 $f_\ell \leq 0 $
 for all $\ell \not\in T$ (we call $F(s)$ a test function).
Then we have the inequality
\[ f_0|X|^2-F(1)|X| \ge \sum_{s\in A(X)} F(s)A_s(X),\]
where $A_s(X)$ and $A(X)$ are defined in \eqref{eq:d_a} and \eqref{eq:A(X)}, respectively. 
\end{lemma}
\begin{proof}
This is an immediate consequence of \cite[Lemma 3.1]{HiraoNozakiTasaka} (see also \cite[Corollary 3.8]{DelsarteGoethalsSeidel77}). 
\end{proof}

We are now in a position to show Theorem \ref{thm:LP_bound}.

\begin{proof}[Proof of Theorem \ref{thm:LP_bound}]
(i) This was shown in \cite[Theorem 3.2]{HiraoNozakiTasaka}, but we repeat the proof.
Let $X'$ be a half set of an antipodal spherical $\{10,4,2\}$-design $X$ of $\mathbb{S}^3$.
From \eqref{eq:antipodal_cardinality} it suffices to show that $|X'|\ge12$.
Let 
\[ F_{2T}(s):= \frac{1}{11264} Q_{10}(s) + \frac{1}{2560} Q_{4} (s) + 
\frac{1}{768} Q_{2} (s) + \frac{3}{1024}.\]
Applying this to Lemma \ref{lem:LP_bound}, we obtain
\begin{equation}\label{eq:F_{2T}_ineq} 
 \frac{3}{1024} |X'|^2 - F_{2T}(1)|X'| \ge \sum_{s\in A(X')} F_{2T}(s)A_s(X').
\end{equation}
Note that the polynomial $F_{2T}$ has the following factorization:
\begin{equation}\label{eq:F_{2T}}
F_{2T}(s)= s^2 \left(s + \frac{1}{2} \right)^2 \left(s - \frac{1}{2} \right)^2 \left( \left(s^2 - \frac78\right)^2 + \frac34\right).
\end{equation}
It follows from \eqref{eq:F_{2T}} that $F_{2T}(s)\ge0$ for any $s\in [-1,1]$.
Since $A(X')\subset (-1,1)$, we obtain the inequality $\sum_{s\in A(X')} F_{2T}(s)A_s(X')\ge0$, and hence,
\[ |X'| \ge \frac{1024}{3} F_{2T}(1)=12.\]
 
(ii) Let $X'$ be a half set of an antipodal spherical $\{14,10,6,4,2\}$-design $X$ of $\mathbb{S}^3$.
Consider
\begin{align*}
F_{2O}(s)&:= \frac{1}{245760} Q_{14}(s) + \frac{1}{135168} Q_{10}(s) + \frac{1}{114688} Q_6(s) + \frac{1}{49152} Q_4(s) 
   + \frac{1}{147456 }Q_2(s) + \frac{1}{8192}.
\end{align*}
With this, Lemma \ref{lem:LP_bound} gives the inequality $ \frac{1}{8192} |X'|^2 - F_{2O}(1)|X'| \ge \sum_{s\in A(X')} F_{2O}(s)A_s(X')$.
From the another expression
\begin{equation}\label{eq:F_{2O}}
F_{2O}(s)=s^2\left( s+\frac12\right)^2  \left( s-\frac12\right)^2 \left(  s^2-\frac12\right)^2 \left(\left(s^2- \frac78 \right)^2 +\frac14 \right),
\end{equation}
we see that $F_{2O}(s)\ge0$ for any $s\in [-1,1]$.
Therefore, one gets the desired inequality 
\[ |X'| \ge 8192 F_{2O}(1)= 24.\]

(iii) 
Let $X'$ be a half set of an antipodal spherical $\{10,8,6,4,2\}$-design $X$ of $\mathbb{S}^3$.
Applying 
\begin{equation}\label{eq:F_{2I}}
\begin{aligned}
F_{2I}(s) &:= 
- \frac{1}{1114112} Q_{16}(s)
- \frac{11}{4915200} Q_{14}(s)
+ \frac{21}{1802240} Q_{10}(s)
+ \frac{17}{491520} Q_{8}(x) \\
&\quad + \frac{11}{163840} Q_{6}(s)
+ \frac{177}{1638400} Q_4(s)
+ \frac{149}{983040} Q_2(s)
+ \frac{3}{16384} \\
&= s^2 \left(s - \frac{\tau^{-1}}{2} \right)^2
\left(s + \frac{\tau^{-1}}{2} \right)^2
\left(s - \frac{\tau}{2} \right)^2
\left(s + \frac{\tau}{2} \right)^2 \left(s^2 - \frac14\right)^2
\left( \frac{6}{5} - s^2 \right)
\end{aligned}
\end{equation}
to Lemma \ref{lem:LP_bound}, we obtain 
$\frac{3}{16384} |X'|^2 - F_{2I}(1)|X'|\ge 0$ (note that the coefficients of $Q_{16}$ and $Q_{14}$ in the Gegenbauer expansion of $F_{2I}$ are negative).
Hence, it holds that
\[ |X'|\ge  \frac{16384}{3} F_{2I}(1) = 60.\]
\end{proof}

\begin{remark}
By finding a test function of degree $17$, which is different from our $F_{2I}(s)$, Andreev \cite[Theorem 1]{Andreev} proved that every spherical $11$-design $X\subset \mathbb{S}^3$ satisfies $|X|\ge120$ (see also \cite[Theorem 6]{BoyvalenkovDanev01}).
\end{remark}

From the above proofs, we are led to the following definition.

\begin{definition}
\begin{itemize}
\item[(i)] An antipodal spherical $\{10,4,2\}$-design $X$ of $\mathbb{S}^3$ is said to be \emph{minimal} when $|X|=24$.
\item[(ii)] An antipodal spherical $\{14,10,6,4,2\}$-design $X$ of $\mathbb{S}^3$ is said to be \emph{minimal} when $|X|=48$.
\item[(iii)] An antipodal spherical $\{10,8,6,4,2\}$-design $X$ of $\mathbb{S}^3$ is said to be \emph{minimal} when $|X|=120$.
\end{itemize}
\end{definition}

Using the expressions \eqref{eq:F_{2T}}, \eqref{eq:F_{2O}} and \eqref{eq:F_{2I}} of our test functions, one can control the set $A(X)$ of inner products for the above minimal antipodal spherical $T$-designs $X$. 

\begin{corollary}\label{cor:angles}
\begin{itemize}
\item[(i)] An antipodal spherical $\{10,4,2\}$-design $X$ of $\mathbb{S}^3$ is minimal if and only if
\[A(X) \subset \left\{-1,\pm \frac12,0\right\}.\]
\item[(ii)] An antipodal spherical $\{14,10,6,4,2\}$-design $X$ of $\mathbb{S}^3$ is minimal if and only if 
\[A(X)\subset \left\{-1,\pm \frac{1}{\sqrt{2}},\pm \frac12,0\right\}.\]
\item[(iii)] An antipodal spherical $\{10,8,6,4,2\}$-design $X$ of $\mathbb{S}^3$ is minimal if and only if
\[A(X)\subset \left\{-1,\pm \frac{\tau^{-1}}{2},\pm \frac{\tau}{2}, \pm \frac12,0\right\}.\]
\end{itemize}
\end{corollary}
\begin{proof}
For a half set $X'$ of an antipodal spherical $\{10,4,2\}$-design $X$ of $\mathbb{S}^3$, it follows from \eqref{eq:F_{2T}_ineq} that the equality $|X'|=12$ holds if and only if $F_{2T}(s)=0$ for all $s\in A(X')$.
Thus, the result follows from the expression in \eqref{eq:F_{2T}}.
The other two cases are obtained in much the same way as above and are therefore omitted.
\end{proof}

\subsection{Proof of Theorem \ref{thm:uniqueness}}

In this subsection, we show the uniqueness of $X=2O$ as a minimal antipodal spherical $\{14,10,6,4,2\}$-design, whose cardinality is $48$. 

\begin{proof}[Proof of Theorem \ref{thm:uniqueness} (ii)] 
Let 
\[ A:= \left\{-1,\pm \frac{1}{\sqrt{2}},\pm \frac12,0\right\}.\]
From Corollary \ref{cor:angles} (ii), the set of the inner products satisfies $A(X)\subset A$ (namely, $X$ is an $A$-code).
Fix an arbitrary point $\boldsymbol{x}_0 \in X$, and we may suppose $\boldsymbol{x}_0=(1,0,0,0)$. 
For $s\in \R$, let 
\[ X_s: =\{ \boldsymbol{x} \in X \mid \langle \boldsymbol{x}, \boldsymbol{x}_0 \rangle =s \}.  \]
Note that $X_{-s}=-X_s$, since $X$ is antipodal.
Each point $\boldsymbol{x}$ of $X_{s}$ has the form $\boldsymbol{x}=(s,x_2,x_3,x_4)$. For such $\boldsymbol{x}$ with $s\neq \pm 1$, we write $\widetilde{\boldsymbol{x}}=\sqrt{1/(1-s^2)} (x_2,x_3,x_4)$ which is on the unit sphere $\mathbb{S}^2$.
For $s \in A(X)\setminus\{-1\}$, the set 
 \[\widetilde{X}_{s}:=\{\widetilde{\boldsymbol{x}} \in \mathbb{S}^2 \mid \boldsymbol{x} \in X_{s}\} \] 
is called the derived code with respect to $\boldsymbol{x}_0$. 
By \cite[Theorem 8.2]{DelsarteGoethalsSeidel77}, each derived code
$\widetilde{X}_s$ is a spherical $3$-design on $\mathbb{S}^2$, as $X$ is a spherical $7$-design and an $A$-code.

Notice that $X$ is distance invariant from 
Theorem 7.4 in \cite{DelsarteGoethalsSeidel77}, namely $ |X_s |$ does not depend on the choice of the initial $\boldsymbol{x}_0 \in X$.
Thus, we have $A_s(X)=|X_s|\times |X|$, where $A_s(X)$ is defined in \eqref{eq:d_a}.
From Lemma \ref{lem:gegenbauer}, one obtains the following system of equations in $|X_s|$: 
\[\sum_{s \in A\cup\{1\} }|X_s| Q_\ell^{(4)}(s) = 0\]
for each $\ell \in \{2, 4, 6\}$ (see \eqref{eq:Q_l^{d}} for the polynomials $Q_\ell^{(4)}(s)$). Together with the conditions $|X_{\pm 1} |=1$ and $|X_s| = |X_{-s}|$, this system of equations is solved as 
\[ |X_0| = 18, \qquad |X_{\pm 1/\sqrt{2}}| = 6, \qquad |X_{\pm 1/2}| = 8.\]
Since a spherical 3-design on $\mathbb{S}^2$ with $6$ points (namely, a tight 3-design) is unique and is a cross polytope (the regular octahedron) up to orthogonal transformations (cf.~\cite{BB09,BZ20}), we see that $\widetilde{X}_{\pm 1/\sqrt{2}}$ is an orthogonal transformation of $\{(\pm 1,0,0),(0,\pm 1,0),(0,0,\pm 1)\}$.
By taking a suitable orthogonal transformation, we may express
\begin{align*}
    X_1&=\{(1,0,0,0)\}, \qquad 
    X_{-1}=\{(-1,0,0,0)\}, \\ 
    X_{1/\sqrt{2}}&= \left\{\left(\frac{1}{\sqrt{2}},\pm \frac{1}{\sqrt{2}},0 ,0\right),
    \left(\frac{1}{\sqrt{2}},0,\pm \frac{1}{\sqrt{2}},0 \right),
    \left(\frac{1}{\sqrt{2}},0,0,\pm \frac{1}{\sqrt{2}}\right)
    \right\}, \\
     X_{-1/\sqrt{2}}&= \left\{\left(-\frac{1}{\sqrt{2}},\pm \frac{1}{\sqrt{2}},0 ,0\right),
    \left(-\frac{1}{\sqrt{2}},0,\pm \frac{1}{\sqrt{2}},0 \right),
    \left(-\frac{1}{\sqrt{2}},0,0,\pm \frac{1}{\sqrt{2}}\right)
    \right\}. 
\end{align*}

We now compute the coordinates of $X_{\pm1/2}$.
From the comment below Definition 8.1 in \cite{DelsarteGoethalsSeidel77}, one has
\[ A(\widetilde{X}_{1/2}) \subset \left\{\frac{4}{3}s -\frac{1}{3} \ \middle| \ s \in A \right \}\cap [-1,1)=\left\{\pm \frac{1}{3}, \frac{2\sqrt{2}-1}{3},-1 \right\}.\]
 Since $\widetilde{X}_{1/2}$ is a spherical $3$-design, it follows from Lemma \ref{lem:gegenbauer} that 
 \begin{equation} \label{eq:0}
 \sum_{\widetilde{\boldsymbol{x}}, \widetilde{\boldsymbol{y}} \in \widetilde{X}_{1/2}} \langle \widetilde{\boldsymbol{x}}, \widetilde{\boldsymbol{y}} \rangle =0.     
 \end{equation}
If \( (2\sqrt{2}-1)/3 \) appears as an inner product in \eqref{eq:0}, then the left-hand side of \eqref{eq:0} is never zero.
This implies that $A(\widetilde{X}_{1/2}) \subset \left\{\pm \frac{1}{3},-1 \right\}$. 
Thus by \cite[Theorem 7.4]{DelsarteGoethalsSeidel77}, the set $Y=\widetilde{X}_{1/2}$ is distance invariant, namely, $|\{\boldsymbol{y}\in Y\mid \langle \boldsymbol{y},\boldsymbol{y}_0\rangle=s\}|$ does not depend on the choice of $\boldsymbol{y}_0\in Y$.
Fix $\boldsymbol{y}_0\in Y$ and let $Y_s=\{\boldsymbol{y}\in Y\mid \langle \boldsymbol{y},\boldsymbol{y}_0\rangle=s\}$.
Note that $|Y_1|=1$.
From Lemma \ref{lem:gegenbauer}, one has
\[\sum_{s\in \{\pm 1,\pm \frac13 \} }|Y_s| Q_\ell^{(3)}(s) = 0, \quad \ell=1,2,3.\]
This system of equations has the following solutions:
\[ |Y_{\pm 1/3}|=3, \qquad |Y_{-1}|=1.\]
Therefore, $Y$ is the cube in $\mathbb{S}^2$ (cf.~\cite{BB09,BZ20}). 
From \cite[Lemma 1]{BoyvalenkovDanev01}, for $\boldsymbol{y} \in Y$ and $\widetilde{\boldsymbol{x}} \in \widetilde{X}_{1/\sqrt{2}}$, we see that
\[ \langle \boldsymbol{y}, \widetilde{\boldsymbol{x}} \rangle \in 
\left\{\sqrt{\frac{8}{3}} s -\frac{1}{\sqrt{3}}\ \middle| \ s \in A\right\} \cap[-1,1)=\left\{\pm \frac{1}{\sqrt{3}}, \sqrt{\frac{2}{3}}-{\frac{1}{\sqrt{3}}}\right\}.\]
Using \cite[Corollary 2]{BoyvalenkovDanev01}, we compute the distance distribution of $Y$ with respect to $\widetilde{\boldsymbol{x}} \in \widetilde{X}_{1/\sqrt{2}}$: 
\[ \left|Y_{1/\sqrt{3}}^{\widetilde{\boldsymbol{x}}}\right|= 4, \qquad \left|Y_{-1/\sqrt{3}}^{\widetilde{\boldsymbol{x}}}\right|=4, \qquad \left|Y_{\sqrt{2/3}-1/\sqrt{3}}^{\widetilde{\boldsymbol{x}}}\right|=0,  \]
where $Y_s^{\widetilde{\boldsymbol{x}}}=\{\boldsymbol{\boldsymbol{y}} \in Y \mid \langle \boldsymbol{y}, \widetilde{\boldsymbol{x}} \rangle = s \}$. This implies that for each $\boldsymbol{x} \in X_{1/\sqrt{2}}$, we have that
\[ |\{\boldsymbol{y} \in X_{1/2} \mid \langle \boldsymbol{y} ,\boldsymbol{x} \rangle = 1/\sqrt{2}\}|=4, 
\qquad 
|\{ \boldsymbol{y} \in X_{1/2} \mid \langle \boldsymbol{y}, \boldsymbol{x}\rangle = 0\}|=4. \]
Using the above expression of $X_{1/\sqrt{2}}$, we can uniquely determine 
\[ X_{1/2}=\left\{\left(\frac{1}{2}, \pm \frac{1}{2}, \pm \frac{1}{2}, \pm \frac{1}{2}\right ) \right\}, \]
where we take all choices of signs. 
$X_{-1/2}$ is obtained by $-X_{1/2}$.

Finally we determine $\boldsymbol{x}=(0,x_2,x_3,x_4)\in X_0$. 
The inner products between elements in $X_0$ and in $X_{1/\sqrt{2}}$ belong to $A$, so $\pm \frac{x_j}{\sqrt{2}} \in A$ for 
$j=2,3,4$. 
This together with $x_2^2+x_3^2+x_4^2=1$ shows $x_j\in\{\pm1,\pm1/\sqrt{2}\}$.
Therefore, $X_0$ should be a subset of 
\[  \left\{(0,\pm 1, 0,0), (0,0,\pm 1, 0),(0,0,0,\pm 1),\left(0,\pm \frac{1}{\sqrt{2}},\pm \frac{1}{\sqrt{2}},0  \right),\left(0,\pm \frac{1}{\sqrt{2}},0,\pm \frac{1}{\sqrt{2}}  \right),\left(0,0,\pm \frac{1}{\sqrt{2}},\pm \frac{1}{\sqrt{2}}  \right) \right\}. \]
This set is size 18, and hence, coincides with $X_0$.

From the coordinates of $X_{1/\sqrt{2}}$, we can uniquely determine the coordinates of $X_s$ for each $s \in A(X)$. This implies the uniqueness of $2O$ as desired. 
\end{proof}

\section{Application to modular forms}\label{sec:order}

\subsection{Orders and modular forms}
Let $F$ be a totally real number field and $R$ its ring of integers.
The $F$-vector space 
\[D(F):=\{x_1+x_2i+x_3j+x_4k\mid x_1,\ldots,x_4\in F\}\]
forms an $F$-algebra, called the \emph{quaternion algebra over $F$}.
A subring $O\subset D(F)$ is an \emph{$R$-order} if there exists a basis $\{b_1,b_2,b_3,b_4\}$ of $D(F)$ such that $O=Rb_1\oplus R b_2\oplus Rb_3\oplus R b_4$. 
Note that the set $N(O)$ of all norm values of $O$ is a subset of $R$ (cf.~\cite[Chapter 10]{Voight2021}).
An $R$-order of $D$ is said to be \emph{maximal} if it is not properly contained in another $R$-order of $D$.

Recall our orders from the introduction.
Let $F_{2T}=\Q,\ F_{2O}=\Q\big(\sqrt{2}\big)$ and $F_{2I}=\Q\big(\sqrt{5}\big)$.
From \eqref{eq:def_exceptional_groups}, we see that each $G\in\{2T,2O,2I\}$ is a subset of $F_G^4$.
Let $R_G$ be the ring of integers of $F_G$, i.e., $R_{2T}=\Z, R_{2O}=\Z[\sqrt{2}], R_{2I}=\Z[\tau]$ with $\tau=\frac{1+\sqrt{5}}{2}$, and $\mathcal{O}_G=\langle G\rangle_{R_G}$ the $R_G$-subalgebra of the quaternion algebra over $F_G$ generated by $G$.
Then, $\mathcal{O}_G$ is a maximal $R_G$-order (cf.~\cite[Chapter 11]{Voight2021}).
Its basis is given in Table \ref{table:orders}.
For $G\in \{2T,2O,2I\}$, we define the map $\iota_G:R_G\rightarrow\Z$ by
\begin{equation}\label{eq:def_iota}
\begin{aligned}
&\iota_{2T}(a)=a, \quad (a\in \Z=R_{2T}),\\
&\iota_{2O}(a+b\sqrt{2})=a, \quad (a+b\sqrt{2}\in \Z+\Z\sqrt{2}=R_{2O}),\\
&\iota_{2I}(a+b\tau)=a, \quad (a+b\tau\in \Z+\Z\tau=R_{2I}).
\end{aligned}
\end{equation}

For $G\in \{2T,2O,2I\}$, let
\[\mathcal{Q}_G(x):=\iota_G(N(x)).\]
For $m\in \N$, we study the following finite subset of $\R^4$:
\[ \mathcal{O}_{G,m}=\{ x\in \mathcal{O}_G\mid \mathcal{Q}_G(x)=m\}.\]
We give explicit formulas for $|\mathcal{O}_{G,m}|$.
For this, let $\sigma_k(m)=\sum_{d\mid m}d^k$ be the divisor function.
We set $\sigma_k(m/2)=0$ if $m$ is odd.

\begin{proposition}\label{prop:counting}
\begin{itemize}
\item[(i)] For $m\in \N$, we have $|\mathcal{O}_{2T,m}| = 24\left(\sigma_1(m)-2\sigma_1\left(\frac{m}{2}\right)\right)>0$. 
Moreover, $\mathcal{O}_{2T,1}=2T$. 
\item[(ii)] For $m\in \N$, we have $|\mathcal{O}_{2O,m}| = 240\left(5\sigma_3(m)-4\sigma_3\left(\frac{m}{2}\right)\right)>0$. 
Moreover, $\mathcal{O}_{2O,1}=2O$. 
\item[(iii)] For $m\in \N$, we have $|\mathcal{O}_{2I,m}| =  240\sigma_3(m)>0$. 
Moreover, $\mathcal{O}_{2I,1}=2I\cup \tau\, 2I$. 
\end{itemize}
\end{proposition}
\begin{proof}
For each $G\in \{2T,2O,2I\}$, we first observe that $\mathcal{Q}_G(x)\in \Z$ for $x\in \mathcal{O}_G$ defines a positive definite integral quadratic form (for background on quadratic forms, see \cite[Chapter 6]{Ogg}).
This implies that the spherical theta function $\theta_{G,P}(z)=\sum_{x\in \mathcal{O}_G} P(x) q^{\mathcal{Q}_G(x)} \ (q=e^{2\pi i z})$ defined in \eqref{eq:theta}, is a modular form (cf.~\cite[\S4.9]{Miyake}, \cite[\S6]{Ogg}).
Then, using the theory of modular forms, we express the generating function $\theta_{G,1}(z)=\sum_{m\ge0}| \mathcal{O}_{G,m} |q^{m}$ in terms of the Eisenstein series $E_2(z)=1-24\sum_{m\ge1}\sigma_1(m)q^{m}$ and $E_4(z)=1+240\sum_{m\ge1}\sigma_3(m)q^{m}$.
For the theory of modular forms for the congruence subgroup, see \cite{DS}.
The quadratic form $\mathcal{Q}_G(x)$ is also used to give a description of $\mathcal{O}_{G,1}$ (we solve the Diophantine equation $\mathcal{Q}_G(x)=1$ with the aid of Mathematica).

For the Hurwitz order $\mathcal{O}_{2T}$, let $x=r_1+r_2 i+ r_3 j+ r_4 \omega\in \mathcal{O}_{2T}$ with $r_1,r_2,r_3,r_4\in \Z$.
Then one obtains
\[\mathcal{Q}_{2T}(x)= r_1^2+r_2^2+r_3^2+r_4^2-r_1r_4-r_2r_4-r_3r_4,\]
which defines a positive definite integral quadratic form of level 2 and discriminant 4.
Hence, the spherical theta function $\theta_{2T,P}(z)$ for $P\in {\rm Harm}_\ell(\R^4)$ is a modular form of weight $2+\ell$ for $\Gamma_0(2)$.
The space of modular forms of weight $2$ for $\Gamma_0(2)$ is 1-dimensional spanned by $2E_2(z)-E_2(2z)$.
Therefore, we have
\[\theta_{2T,1}(z)=2E_2(2z)-E_2(z) =1+24q+24q^2+96q^3+24q^4+\cdots. \]
By finding integral solutions to the equation $r_1^2+r_2^2+r_3^2+r_4^2-r_1r_4-r_2r_4-r_3r_4=1$, we obtain $\mathcal{O}_{2T,1}=2T$.
Below is an example of Mathematica code to compute the set $\mathcal{O}_{2T,1}$:

\verb|Solve[r1^2+r2^2+r3^2+r4^2-r1*r4-r2*r4-r3*r4==1, {r1,r2,r3,r4}, Integers]|

\

For the case $\mathcal{O}_{2O}$, using the integral basis of $\Z[\sqrt{2}]$, we write $x=(y_1+z_1\sqrt{2})+(y_2+z_2\sqrt{2})\alpha +(y_3+z_3\sqrt{2})\beta+(y_4+z_4\sqrt{2})\alpha \beta \in \mathcal{O}_{2O}$ with $y_j,z_j\in \Z$.
Then we have
\begin{align*}
 \mathcal{Q}_{2O}(x) &= y_{1}^2+y_{1} y_{4}+2 y_{1} z_{2}+2 y_{1}
   z_{3}+y_{2}^2+y_{2} y_{3}+2 y_{2} z_{1}+2 y_{2}
   z_{4}+y_{3}^2+2 y_{3} z_{1}\\
 &+2 y_{3} z_{4}+y_{4}^2+2
   y_{4} z_{2}+2 y_{4} z_{3}+2 z_{1}^2+2 z_{1} z_{4}+2
   z_{2}^2+2 z_{2} z_{3}+2 z_{3}^2+2 z_{4}^2,
\end{align*}
which is positive definite.
The level and the discriminant of the above quadratic form is 2 and 16, respectively.
Hence, $\theta_{2O,P}(z)$ for $P\in {\rm Harm}_\ell(\R^4)$ is a modular form of weight $4+\ell$ for $\Gamma_0(2)$.
Note that the space of modular forms of weight $4$ for $\Gamma_0(2)$ is 2-dimensional.
Due to the Strum bound \cite[Corollary 9.20]{Stein}, expressing $\theta_{2O,1}(z)$ in terms of the Eisenstein series requires the coefficient of $q$, namely, $|\mathcal{O}_{2O,1}|$.
Now, let the computer calculate integer solutions to $\mathcal{Q}_{2O}(x)=1$ to get an explicit description of $\mathcal{O}_{2O,1}$. 
The result is $\mathcal{O}_{2O,1}=2O$. 
Hence, $|\mathcal{O}_{2O,1}|=48$ and we get
\[ \theta_{2O,1}(z)=5 E_4(z)-4 E_4(2z) = 1+48q+624q^2+1344q^3+5232q^4+\cdots.\]

For the case $\mathcal{O}_{2I}$, we also write $x=(y_1+z_1\tau)+(y_2+z_2\tau)i +(y_3+z_3\tau)\zeta +(y_4+z_4\tau)i\zeta \in \mathcal{O}_{2I}$ with $y_j,z_j\in \Z$.
It holds that
\begin{align*} 
\mathcal{Q}_{2I}(x) &= y_{1}^2+y_{1} y_{4}+y_{1} z_{3}-y_{1}
   z_{4}+y_{2}^2-y_{2} y_{3}+y_{2} z_{3}+y_{2}
   z_{4}+y_{3}^2+y_{3} z_{1}\\
   &+y_{3} z_{2}+y_{4}^2-y_{4}
   z_{1}+y_{4} z_{2}+z_{1}^2+z_{1} z_{3}+z_{2}^2+z_{2}
   z_{4}+z_{3}^2+z_{4}^2.
   \end{align*}
This is a positive definite integral quadratic form of level 1 and discriminant 1.
Thus, $\theta_{2I,P}(z)$ for $P\in {\rm Harm}_\ell(\R^4)$ is a modular form of weight $4+\ell$ for ${\rm SL}_2(\Z)$.
Since the underlying vector space is 1-dimensional, we get
\[ \theta_{2I,1}(z)=E_4(z) = 1+240q+2160q^2+6720q^3+17520q^4+\cdots.\]
One can check that the set $\mathcal{O}_{2I,1}$ of integer solutions to the equation $\mathcal{Q}_{2I}(x)=1$ has 240 points. Moreover, we obtain $\mathcal{O}_{2I,1}=2I \cup \tau\, 2I$, meaning that the set $\mathcal{O}_{2I,1}$ consists of two copies of $2I$.
This completes the proof.
\end{proof}

\begin{remark}
As the formula $|\mathcal{O}_{2I,m}| = 240\sigma_3(m)$ suggests, the $\Z[\tau]$-order $\mathcal{O}_{2I}$ has a connection with the $E_8$ lattice (an even integral unimodular lattice in $\R^8$).
Let $\kappa_4 : D(\Q(\tau))\rightarrow \Q^8$ be the $\Q$-linear map defined by 
\[ \kappa_4\big( (a_1+b_1\tau)+(a_2+b_2\tau)i+(a_3+b_3\tau)j+(a_4+b_4\tau)k\big)=(a_1,b_1,a_2,b_2,a_3,b_3,a_4,b_4).\]
Then, the image $\kappa_4\big(\mathcal{O}_{2I}\big)$ is isometric to the $E_8$ lattice.
Furthermore, the map $\kappa_4$ gives a bijection from $\mathcal{O}_{2I,m}$ to the $2m$-shell of the $E_8$ lattice.
This is because for $x=(a_1+b_1\tau)+(a_2+b_2\tau)i+(a_3+b_3\tau)j+(a_4+b_4\tau)k\in\mathcal{O}_{2I}$, it holds that
\[\iota_{2I}(N(x))=a_1^2+b_1^2+\cdots+a_4^2+b_4^2=\langle \kappa_4(x),\kappa_4(x)\rangle,\] 
where we have used $\tau^2=\tau+1$.
For more details, see \cite[Chapter 8]{ConwaySloan}. 
\end{remark}

\subsection{Proof of Theorem \ref{thm:dim=0}}

We first prove \eqref{eq:G_acts_on_O_{G,m}}.
For this, recall that the multiplication by a quaternion yields an orthogonal transformation (Lemma \ref{lem:quaternion_orthogonal}).

\begin{proposition}\label{prop:O_strength}
Let $G\in\{2T,2O,2I\}$. 
For all $m\in \N$, we have $T(G)\subset T\big(\mathcal{O}_{G,m}\big)$.
In particular, the equality $T(G)=T\big(\mathcal{O}_{G,1}\big)$ holds.
\end{proposition}
\begin{proof}
Let $m\in \N$.
For $\varepsilon\in G$ and $x\in \mathcal{O}_{G,m}$, since $N(x\varepsilon)=N(x)$, it holds that $x\varepsilon \in \mathcal{O}_{G,m}$.
This shows that $G$ acts on the set $\mathcal{O}_{G,m}$. 
Hence, we have a $G$-orbit decomposition 
\begin{equation}\label{eq:orbit_decomp} 
\mathcal{O}_{G,m} = \bigsqcup_{x\in S_m} xG
\end{equation}
of $\mathcal{O}_{G,m}$ as a disjoint union of $\{xG \mid x\in S_m\}$ for some $S_m\subset \mathcal{O}_{G,m}$.
By Lemma \ref{lem:quaternion_orthogonal}, the $G$-orbit $xG$ is an orthogonal transformation of $\sqrt{N(x)} G$.
It is well known that, for $P\in {\rm Harm}_\ell (\R^4)$ and $g\in O(\R^4)$, we have $gP\in {\rm Harm}_\ell (\R^4)$, where $(gP)(x):=P(g(x))$.
Therefore, $T(G)= T(xG)$ holds for any $x\in S_m$.
Finally, the result follows from the fact that $T(X_1\sqcup X_2)$ contains $ T(X_1)\cap T(X_2)$.

Let us turn to the case $m=1$.
From Proposition \ref{prop:counting}, for $G\in \{2T,2O\}$, we have $G=\mathcal{O}_{G,1}$.
Thus, $T(G)=T\big(\mathcal{O}_{G,1}\big)$.
For the case $G=2I$, since $T(2I)\subset T\big(\mathcal{O}_{2I,1}\big)$, it suffices to show $T(2I)\supset T\big(\mathcal{O}_{2I,1}\big)$.
Suppose $\ell\in T\big(\mathcal{O}_{2I,1}\big)$.
By Proposition \ref{prop:counting} (iii), for all $P\in {\rm Harm}_\ell(\R^4)$ one has
\[ 0=\sum_{x\in \mathcal{O}_{2I,1}} P(x)=\sum_{x\in 2I} P(x) + \sum_{x\in 2I} P(\tau x)=(1+\tau^\ell ) \sum_{x\in 2I} P( x).\]
Thus, $\ell \in T(2I)$.
This completes the proof.
\end{proof}

We are now in a position to prove Theorem \ref{thm:dim=0}.

\begin{proof}[Proof of Theorem \ref{thm:dim=0}]
We first prove that $\dim_\C \Theta(G,\ell)=0 \Rightarrow \ell \in T(G)$.
Suppose $\ell \not\in T(G)$.
Then, by Lemma \ref{prop:O_strength}, $\ell \not\in T\big(\mathcal{O}_{G,1}\big)$.
This shows that there exists $P\in {\rm Harm}_\ell(\R^4)$ such that $\theta_{G,P}(z)\neq0$.
Hence, $\dim_\C\Theta(G,\ell)\ge1$.

Let us turn to the proof of $ \ell \in T(G)\Rightarrow \dim_\C \Theta(G,\ell)=0$.
For $\ell \in T(G)$, by Lemma \ref{prop:O_strength}, $\ell \in T\big(\mathcal{O}_{G,m}\big)$ holds for all $m\in\N$.
Hence, $\theta_{G,P}=0$ for all $P\in {\rm Harm}_\ell (\R^4)$.
Thus, $\dim_\C\Theta(G,\ell)=0$.
We have done.
\end{proof}

\begin{remark}\label{rem:HNT}
In our previous paper \cite[\S6]{HiraoNozakiTasaka}, as an application of the uniqueness, we gave a decomposition of the $m$-shell of the $D_4$ lattice in terms of a disjoint union of the orthogonal transformations (up to scaler) of the $D_4$ root system.
We have the same structure: For any $m>0$, there exists a finite subset $U_m$ of the orthogonal transformation group $O(\R^4)$ such that 
\[\left\{ \frac{x}{\sqrt{N(x)}} \ \middle| \  x\in \mathcal{O}_{G,m}\right\}=\bigsqcup_{g\in U_m} g (G).\]
This is because the orbit $xG$ in \eqref{eq:orbit_decomp} is an orthogonal transformation of $\sqrt{N(x)} G$.
\end{remark}

\subsection{More on the space $\Theta(G,\ell)$}

In this last subsection, we first give an upper bound of the dimension of the space $\Theta(G,\ell)$, by using the theory of invariant polynomials.
Then, numerical basis of the non-zero space $\Theta(G,\ell)$ are provided for some $\ell$.

\subsubsection{Invariant polynomials}
For a subgroup $G$ of $\mathbb{H}_1$, denote the $G$-invariant subspace of ${\rm Harm}_{\ell} (\R^4)$ by 
\[{\rm Harm}_\ell (\R^4)^G :=\{ P\in {\rm Harm}_\ell (\R^4) \mid \varepsilon P=P\ \mbox{for all}\ \varepsilon \in G\},\]
where we set $(\varepsilon P)(x):=P(x M_\varepsilon)$ for $x\in\R^4$ and $M_\varepsilon$ is the matrix defined in \eqref{eq:M_x}.
If we identify $\R^4$ with $\mathbb{H}$ as in Lemma \ref{lem:quaternion_orthogonal}, then $(\varepsilon P)(x)=P(\varepsilon x)$.

\begin{proposition}\label{prop:upper_bound}
For $G\in \{2T,2O,2I\}$ and $\ell\in \N$, we have 
\[ \dim_\C \Theta(G,\ell) \le \dim_\R {\rm Harm}_\ell (\R^4)^G.\]
\end{proposition}
\begin{proof}
From the representation theory of finite groups, we obtain
\[ {\rm Harm}_\ell (\R^4)= {\rm Harm}_\ell (\R^4)^G\oplus {\rm Span}_{\R}\{P- \varepsilon P \mid P\in {\rm Harm}_\ell (\R^4), \varepsilon\in G\}.\]
Then, the result follows from 
\begin{align*}
\theta_{G,{P-\varepsilon P}}(z)&=\sum_{x\in \mathcal{O}_G} \big( P(x)-P(\varepsilon x)\big) q^{\mathcal{Q}_G(x)}=\theta_{G,P}(z)-\sum_{x\in \varepsilon \mathcal{O}_G} P(x)q^{\mathcal{Q}_G(\varepsilon^{-1}x)}\\
&=\theta_{G,P}(z)-\sum_{x\in \mathcal{O}_G} P(x)q^{\mathcal{Q}_G(x)}=0,
\end{align*}
where we have used the fact that $G$ acts on $\mathcal{O}_G$ and that $N(\varepsilon x)=N(x)$ for $\varepsilon\in \mathbb{H}_1$ and $x\in \mathbb{H}$.
\end{proof}

The generating series of the dimension of ${\rm Harm}_\ell (\R^4)^G$ is called the \emph{harmonic Molien series of $G$}.
We denote it by
\[ \Psi_G^H(u):=\sum_{\ell\ge0}  \dim_\R {\rm Harm}_\ell (\R^4)^G u^\ell.\]
From \cite[Theorem 4.3.2]{Smith} and $ {\rm Hom}_\ell(\R^4)^G = {\rm Harm}_\ell (\R^4)^G\oplus  N(x)  {\rm Hom}_{\ell-2}(\R^4)^G$, we obtain
\[  \Psi_G^H(u)= \frac{1}{|G|} \sum_{\varepsilon \in G} \frac{1-u^2}{\det (I-uM_\varepsilon)}.\]
Using this, we can compute $\Psi_G^H(u)=\sum_{\ell\ge0}d_{G,\ell}u^\ell$ for $G\in \{2T,2O,2I\}$ as follows:
\begin{align*}
&\Psi_{2T}^H(u) = \frac{f_{2T}(u)-f_{2T}(u^{-1})u^{18}}{(1-u^4)^2(1-u^6)^2},\\
&\Psi_{2O}^H(u) = \frac{f_{2O}(u) - f_{2O}(u^{-1})u^{26}}{(1-u^6)^2(1-u^8)^2},\\ 
&\Psi_{2I}^H(u) = \frac{f_{2I}(u)-f_{2I}(u^{-1})u^{34}}{(1+u^2)^2(1-u^6)^2(1-u^{10})^2},
\end{align*}
where we set
\begin{align*}
f_{2T}(u)&=1 - 2 u^4 + 5 u^6 + 10 u^8,\\
f_{2O}(u)&=1 - 2 u^6 + 7 u^8 + 14 u^{12},\\
f_{2I}(u)&=1 + 2 u^2 +u^4 - 2 u^6 - 4 u^8 - 4 u^{10} + 10 u^{12} + 26 u^{14} +  18 u^{16}. 
\end{align*}

\begin{center}
\begin{tabular}{c|cccccccccccccccccccc}
$\ell$ & 2&4 &6 &8&10&12&14&16&18&20&22&24 \\ \hline
$d_{2T,\ell}$  &0 & 0 & 7 & 9 & 0 & 26 & 15 & 17 & 38 & 42 & 23 & 75\\ \hline
$d_{2O,\ell}$ &  0 & 0 & 0 & 9 & 0 & 13 & 0 & 17 & 19 & 21 & 0 & 50 \\ \hline
$d_{2I,\ell}$ &  0 & 0 & 0 & 0 & 0 & 13 & 0 & 0 & 0 & 21 & 0 & 25\\ \hline
\end{tabular}
\end{center}

We believe that the inequality in Proposition \ref{prop:upper_bound} becomes an equality only when $\dim_\C \Theta(G,\ell)=0$.
For the (conjectural) dimension of the space $\Theta(G,\ell) $, see the subsequent subsections.

\subsubsection{Case $G=2T$}
Let $\Lambda$ be the $D_4$ lattice generated by the $D_4$ root system.
This lattice is unique as an even integral lattice of level 2 (see e.g., \cite[Theorem 7.2]{HiraoNozakiTasaka}).
From Remark \ref{rem:D_4}, we see that $\Lambda=(1+i)\mathcal{O}_{2I}$.
Thus, the space $\Theta(2T,\ell)$ coincides with the $\C$-vector space spanned by the spherical theta function of the $D_4$ lattice.
It is generated by newforms of weight $2+\ell$ for $\Gamma_0(2)$ (see \cite[Remark 7.3]{HiraoNozakiTasaka}), namely, we have $\Theta(2T,\ell)=S_{2+\ell}^{\rm new}(2)$ the $\C$-vector space spanned by newforms of weight $2+\ell$ for $\Gamma_0(2)$.
Let $S_k(N)$ denote the $\C$-vector space spanned by cusp forms with real Fourier coefficients of weight $k$ for $\Gamma_0(N)$.
It is known that $S_k(1)$ and $S_k(2)$ have a basis whose Fourier coefficients are rational numbers.
Thus, from the dimension formula (cf.~\cite[Section 3.5]{DS}), we have $\sum_{k\ge0}\dim_\C S_k(1)u^k = \frac{u^{12}}{(1-u^4)(1-u^6)}$ and $\sum_{k\ge0}\dim_\C S_k(2)u^k = \frac{u^{8}}{(1-u^2)(1-u^4)}$.
Since $\dim_\C S_k^{\rm new}(2)= \dim_\C S_k(2)-2\dim_\C S_k(1)$, we obtain the dimension formula for $\Theta(2T,\ell)$ as follows:
\[ \sum_{\ell\ge0} \dim_\C \Theta(2T,\ell) u^\ell =\frac{1+u^{12}}{(1-u^6)(1-u^8)}.\] 

\subsubsection{Case $G=2O$}
From Theorems \ref{thm:harmonic_strength} and \ref{thm:dim=0}, we have $\dim_\C \Theta(2O,2\ell)=0$ if and only if $2\ell\in\{22,14,10,6,4,2\}$.
The first non-trivial space $\Theta(2O,8)$ is a subspace of cusp forms of weight 12 for $\Gamma_0(2)$.
One can observe that $\Theta(2O,8)$ is 1-dimensional spanned by
\[ \Delta(z)+64 \Delta(2z) = q+40q^2+252q^3-3008q^4+4830q^5+\cdots,\]
where $\Delta(z)=q\prod_{n\ge1}(1-q^n)^{24}$ is the unique cusp form of weight 12 for ${\rm SL}_2(\Z)$.
Moreover, we have
\begin{align*}
\Theta(2O,12)&=\C(q-3368q^2+58092q^3-268736q^4+\cdots),\\
\Theta(2O,16)&=\C(q+39880q^2+1279452q^3-1999808q^4-174409410q^5+\cdots),\\ 
\Theta(2O,18)&=\C(q-79024q^2+5687604q^3+103384576q^4-1438933770q^5+\cdots),\\
\Theta(2O,20)&=\C(q+200632q^2+57854412q^3-580869056q^4 + 12385840110 q^5- 70670680416 q^6+\cdots).
\end{align*}

We may expect that
\[ \sum_{\ell\ge0} \dim_\C \Theta(2O,\ell) u^\ell \stackrel{?}{=} 1+ \frac{u^8+u^{12}-u^{14}}{(1-u^6)(1-u^8)}=\frac{1+u^{18}}{(1-u^8)(1-u^{12})}.\]

\subsubsection{Case $G=2I$}
Theorems \ref{thm:harmonic_strength} and \ref{thm:dim=0} imply that $\dim_\C \Theta(2I,2\ell)=0$ if and only if 
\[2\ell\in\{58,46,38,34,28,26,22,18,16,14,10,8,6,4,2\}.\]
The first nontrivial space $\Theta(2I,12) $ is 1-dimensional, generated by $ E_4(z)\Delta(z)$.
One can observe that
\begin{align*}
\Theta(2I,20) &=\C(41q + 719736q^2 - 26525268q^3 + 1272600128q^4+\cdots),\\
\Theta(2I,24) &=\C(23q + 2267928q^2 + 438932196q^3 +\cdots),\\
\Theta(2I,30) &=\C(51q + 107369712q^2 - 33500700684q^3 +\cdots),\\
\Theta(2I,32)&=\C(2207q + 11194301112q^2 + 953165204964q^3 +\cdots), \\
\Theta(2I,36)&=\C(107q + 3571780392q^2 - 1430566003836q^3 + \cdots).
\end{align*}

We may expect that
\[ \sum_{\ell\ge0} \dim_\C \Theta(2I,\ell) u^\ell \stackrel{?}{=}\frac{1+u^{30}}{(1-u^{12})(1-u^{20})}.\]

\section*{Acknowledgments}
This work is partially supported by
JSPS KAKENHI Grant Number 19K03445, 20K03736, 23K03034, 24K06688 and 24K06871.

\section*{Statements and Declarations}
We wish to confirm that there are no known conflicts of interest associated with this publication and there has been no significant financial support for this work that could have influenced its outcome.

The data that support the findings of this study are available from the corresponding author, K.T. upon reasonable request.

\appendix
\section{Errata for our previous paper}
On this occasion, we correct the (typographical) errors in our previous paper \cite{HiraoNozakiTasaka}. We would like to express our gratitude to Professor Akihiro Munemasa for pointing them out.

\begin{itemize}
\item On page 11, in lines 1 and 4, ``$y \in C_8$" should be corrected to ``$y \in \big(\sigma'\big)^{-1}\left(X_{\frac12}\right)$".
Additionally, in line 4 of the same page, ``This implies that $C_8$" should be corrected to ``This implies that $\big(\sigma'\big)^{-1}\left(X_{\frac12}\right)$."
\item On page 14, in the proof of Theorem 7.2, ``Theorem 3.2" in line 7 should be corrected to ``Theorem 3.3."
\item On page 15, just before Theorem 8.1, the sentence
``For a finite set $X\subset \mathbb{S}^{d-1}$, we say that $T\subset \N$ is the harmonic strength of $X$ if $X$ is not a spherical $T'$-design for any $T\subsetneq T'\subset \N$."
should be corrected to
``For a spherical $T$-design $X\subset \mathbb{S}^{d-1}$, we say that $T\subset \N$ is the harmonic strength of $X$ if $X$ is not a spherical $T'$-design for any $T\subsetneq T'\subset \N$."
\item On page 16, the congruence $\frac13 \sum_{\boldsymbol{x}\in (D_4)_{2p}} P_6(\boldsymbol{x}) \equiv 32 p^2(1+p) \bmod 3$ in the proof of Theorem 8.2 needs to be explained carefully:
Let $N$ be a subgroup of $W({\bf F}_4)$ given in the proof of Theorem 6.1 in page 12 and $U_p$ a subset of $ (D_4)_{2p}$ such that $ (D_4)_{2p}=\bigsqcup_{\boldsymbol{x}\in U_p} \boldsymbol{x}^N$. Since $P_6$ is $W({\bf F}_4)$-invariant, we obtain
\[\frac{1}{3} \sum_{\boldsymbol{x} \in (D_4)_{2p}} P_6 (\boldsymbol{x})
=\frac{1}{3} 24 \sum_{\boldsymbol{x} \in U_p} P_6 (\boldsymbol{x})
=8 \sum_{\boldsymbol{x} \in U_p} P_6 (\boldsymbol{x})
\equiv  8 (1+p) (2 p)^2 =32 p^2(1+p) \bmod 3.\]
\end{itemize}


\end{document}